\documentclass{amsart}
\usepackage[mathscr]{eucal}
\usepackage{amssymb, amsmath,array, amscd}
\usepackage{url}

\input xy
\xyoption{all}

\theoremstyle{change}
\newtheorem{thm}{Theorem}[section]
\newtheorem{THM}{Theorem}

\newtheorem{prop}[thm]{Proposition}

\newtheorem{lemma}[thm]{Lemma}

\newtheorem*{conjecture}{Conjecture}

\theoremstyle{definition}

\newtheorem{remark}[thm]{Remark}

\numberwithin{equation}{section}

\begin{document}
\title{The Characteristic variety of a generic  foliation }

\author{Jorge  Vit\'{o}rio  Pereira}
 \address{Jorge Vit\'orio Pereira \\ IMPA, Estrada
 Dona  Castorina, 110\\
 22460-320, Rio de Janeiro, RJ, Brazil}
 \email{jvp@impa.br}

\subjclass[2010]{Primary: 37F75, 16S32; Secondary: 37J30, 32C38.}
\thanks{}

\keywords{characteristic foliation, invariant variety, $\mathcal D$-modules.}
\begin{abstract}
We confirm a conjecture of  Bernstein-Lunts which predicts that the characteristic variety
 of a generic polynomial vector field has no homogeneous involutive subvarieties
besides the zero section and subvarieties of fibers over singular points.
\end{abstract}

\maketitle

\tableofcontents

\section{Introduction}

\subsection{Foliations}Let $\mathcal F$ be a one-dimensional singular holomorphic foliation on a smooth projective variety $X$.
The {\bf characteristic variety} $\mathrm{ch}(\mathcal F)$ of $\mathcal F$ is the irreducible subvariety of $E(T^*X)$, the total space
of the cotangent bundle of $X$, with fiber over a non-singular point $x \in X_0  = X \setminus \mathrm{sing}(\mathcal F)$  equal to
the $1$-forms at $x$ which vanish on $T_x \mathcal F$.  More succinctly, if  $N^* \mathcal F$ is the conormal sheaf of $\mathcal F$
 then its restriction at  $X_0$ is a vector sub-bundle of $T^* X_0$ and we can write
\[
\mathrm{ch}(\mathcal F)  = \overline { E( N^* \mathcal F_{| X_0}) } \,
\]
where the closure is taken in $E(T^*X) \supset E(T^*X_0)$.

\smallskip

Clearly $\mathrm{ch}(\mathcal F)$ is a hypersurface of $E(T^*X)$. If $\omega$ is the  non-degenerate $2$-form
which induces the canonical symplectic structure on $T^*X$ then its restriction to $\mathrm{ch}(\mathcal F)$ induces
a one-dimensional foliation $\mathcal F^{(1)}$ on (the smooth locus of) $\mathrm{ch}(\mathcal F)$ which will be called
the {\bf first prolongation of } $\mathcal F$.

\smallskip

In this work we are interested in the subvarieties of $\mathrm{ch}(\mathcal F)$ invariant by $\mathcal F^{(1)}$ when $\mathcal F$
is sufficiently general. For no matter which $\mathcal F$ there is always at least one subvariety of  $\mathrm{ch}(\mathcal F)$
invariant by $\mathcal F^{(1)}$: the zero section of $T^*X$. If the singular set of  $\mathcal F$ is non-empty but of dimension zero then
the fibers over it, and some subvarieties of these fibers, are also left invariant by $\mathcal F^{(1)}$.

We will say that $\mathrm{ch}(\mathcal F)$ is a {\bf quasi-minimal characteristic variety} if  (a) $\mathcal F$ has isolated singularities; and
(b) every  irreducible
homogeneous ( on the fibers of $\mathrm{ch}(\mathcal F) \to X$ ) subvariety  of $ \mathrm{ch}(\mathcal F)$  left invariant by $\mathcal F^{(1)}$ is either the zero section, or a subvariety of a fiber over the singular
set of $\mathcal F$, or the whole $\mathrm{ch}(\mathcal F)$.

\begin{THM}\label{T:1}
Let $ X$ be a smooth projective variety,  $\mathcal L$ an ample line bundle over it, and $k \gg 0$ a sufficiently large
integer. If $ \mathcal F \in \mathbb P H^0(X, TX \otimes \mathcal L^{\otimes k})$ is a very generic foliation then
 $\mathrm{ch}(\mathcal F)$ is a quasi-minimal characteristic variety.
\end{THM}

 In the statement of the theorem above and throughout, by a very generic point of  a given variety we mean a point outside
 a countable union of  Zariski closed subvarieties. The expression generic point  will be reserved to
 points outside a finite union of Zariski closed subvarieties.

Although Theorem \ref{T:1} can be thought as a natural development of Jouanolou's Theorem  and its subsequent
 generalizations,  see \cite{CP} and references therein, it is motivated by a problem coming from the
 representation theory of Weyl algebras that
we  briefly review below.

\subsection{Weyl algebras}

Let $A_n$ be the $n$-th Weyl algebra over $\mathbb C$, that is $A_n$ is the algebra of $\mathbb C$-linear differential operators
on the polynomial ring $\mathbb C[ x_1, \ldots, x_n]$. A basic invariant of an irreducible $A_n$-module $M$ is its Gelfand-Kirillov dimension $GK\dim M$.
After Bernstein  this invariant is subject to the inequality $GK\dim M \ge n$ and equality holds true for important  classes of irreducible $A_n$-modules, see
\cite{B}. If $GK \dim M =n$ then $M$ is, by definition, a holonomic $A_n$-module

\smallskip

For some time, some  believed  that every irreducible $A_n$-module $M$ is holonomic. In 1985 Stafford came up with examples of $A_n$-modules
of particularly simple form and having Gelfand-Kirillov dimension equal to $2n-1$. His examples are of the form $A_n/ I A_n$ where $I$ is a   principal left ideal
generated by an  element of the form $\xi + f$ where $\xi$ is a polynomial vector field  and $f$ is polynomial, see \cite{St}. For those not familiar
with the Gelfand-Kirillov dimension it is useful to remark that when   $I$ is a  principal maximal left ideal then $GK \dim A_n/ I A_n = 2n -1$, and the
search of examples of non-holonomic $A_n$-modules can be reduced to search of principal maximal left ideals of $A_n$.

\smallskip

Stafford's examples are explicit and his arguments are purely algebraic.  In \cite{BL},  Bernstein and Lunts
present two geometrically oriented approaches to construct principal maximal left ideals of $A_n$, and implement them for the second Weyl algebra.
In rough terms, their strategy rely on the the study of a natural foliation defined on the characteristic varieties  of the module. More specifically
they relate the maximality of the ideal to the non-existence of proper invariant subvarieties of this foliation. To define
a characteristic variety for a $A_n$-module, a filtration of $A_n$ has to be fixed and their two approaches are determined by
the choice of two different filtrations.

In the first approach they look  at the Bernstein filtration of  $A_n$, the $i$-th piece $A_n^i$ consists of polynomials in $\{ x_1, \ldots, x_n, \partial_{x_1}, \ldots,  \partial_{x_n} \}$  of
degree at most $i$. The corresponding symbol maps are
\[
\sigma_k : A_n^{k} \longrightarrow \frac{A_n^k} {A_n^{k-1}} \simeq \mathbb C_k [ x_1, \ldots, x_n, y_1, \ldots, y_n] \, .
\]
They proved that if $n=2$, $k \ge 4$ and $P \in \mathbb C_k [ x_1, \ldots, x_n, y_1, \ldots, y_n]$ is a very generic polynomial then
each  operator $d \in A_n^k$ satisfying $\sigma_k(d) = P$ generates a maximal left ideal of $A_n^{k}$. Still under the assumption
that $k \ge 4$, Lunts extends the above result to arbitrary $n \ge 2$ in \cite{Lu}. For $k = 3$ and $n\ge 2$ the very same statement
has been proved by McCune \cite{Mc}. All these results, in contrast with Stafford's, do not exhibit explicit examples
  of non-holonomic $A_n$-modules but instead prove that they are generic in the above sense. For an algorithm to produce explicit examples of the above form
  for $n=2$ and
  its implementation see \cite{AC}. \marginpar[Talvez mencionar que \cite{AC} prova a exist\^{e}ncia para $n=2$ de um aberto \ldots]

\smallskip

In their other approach, Bernstein and Lunts look at the standard filtration of $A_n$. Now the $i$-th piece
corresponds to differential operators of order $\le i$.
If $\xi$ is a polynomial vector field, $f$ a polynomial and  $I = <  \xi + f > $ then the characteristic variety of  $A_n/IA_n$
 coincides the characteristic variety of the foliation $\mathcal F_{\xi}$ as defined
in the previous section.
If $\mathcal F_{\xi}$
has a quasi-minimal characteristic variety then according to  \cite[Proposition 6]{BL} there exists $f \in \mathbb C[x_1, \ldots, x_n]$ for which
$I= < \xi + f > $ is maximal. While they do show that  a generic $\xi$ of degree $\ge 2$ on $\mathbb C^2$ has this property, they
leave the general case as a conjecture, see \cite[\S 4.2]{BL}.

\begin{conjecture}[{\bf Bernstein-Lunts}]\label{Conj:BL}
Let $n \ge 2$ and  $\xi$ be a very generic polynomial  vector field on $\mathbb C^n$ with coefficients of degree $\ge 2$.
Then $\mathrm{ch}(\mathcal F_{\xi})$ is a quasi-minimal characteristic variety.
\end{conjecture}

The three dimensional case of the conjecture has been proved recently by Coutinho \cite{Coutinho}.  In this paper we will
settle the general case.

\begin{THM}\label{T:2}
Bernstein-Lunts conjecture holds true.
\end{THM}

Even when specialized to $n=3$, our proof is very  different from the one of Coutinho.

\subsection{Acknowledgements} I am  grateful to  S. C. Countinho for, since 2003,  bringing periodically to my attention  Bernstein-Lunts conjecture.

\section{Characteristic varieties and prolongations}

\subsection{Characteristic variety}\label{SS:CV}

Let $X$ be a quasi-projective manifold and $\mathcal F$ be a foliation on $X$ with
cotangent bundle $\mathcal L$, that is $\mathcal F = [ \xi ] \in \mathbb P H^0(X, TX \otimes \mathcal L)$
with the representative $\xi$ having no  divisorial components in its singular set. As in the introduction
set $X_0 =  X \setminus \mathrm{sing}(\mathcal F)$.

Contraction with the twisted vector field $\xi$ determines a morphism of $\mathcal O_X$-modules
\[
T^*X \longrightarrow \mathcal L
\]
whose kernel is $N^* \mathcal F$,  the conormal sheaf of $\mathcal F$. At points $x \in X_0$ the sheaf  $N^* \mathcal F$ is clearly locally free, but
it is not locally free in general. For example it is never locally free at an isolated singularity of $\mathcal F$ as
one can promptly verify. Nevertheless, the restriction of $N^* \mathcal F$ at $X_0$  determines a subbundle of $T^* X_0$ of corank one.
As mentioned in the introduction $\mathrm{ch}(\mathcal F)$ is defined as the closure in $E(T^*X)$ of $E(N^* \mathcal F_{|X_0})$. We will
use $\pi$ to denote  the natural projection  $\pi : E(T^*X) \to X$ as well as its restriction  $\pi : \mathrm{ch}(\mathcal F) \to X$.

\medskip

If $(x_1, \ldots, x_n)$ are local coordinates at a open subset $U\subset X$ then  the vector fields
$\{ \partial_{x_i} = \frac{\partial}{\partial x_i} \}$  can be thought as linear coordinates on $T^* U$: the
value of $\partial_{x_i}$ at a $1$-form $\omega \in T^* U$ is given by the contraction $\omega(\partial_{x_i})$.
Thus, if we set $y_i = \partial_{x_i}$ then $(x_1, \ldots, x_n, y_1, \ldots, y_n)$ are global coordinate functions
for $T^*U$. In particular, if $\xi = \sum a_i \partial_{x_i}$ then
\[
\mathrm{ch}( \mathcal F )_{|\pi^{-1}(U)} = \left\{ \sum a_i y_i = 0 \right\}   \, .
\]

\medskip

The singular set of $\mathrm{ch}(\mathcal F)$ is contained in $\pi^{-1}(\mathrm{sing}(\mathcal F))$ and contains
$\pi^{-1}(\mathrm{sing}(\mathcal F)) \cap X$, where $X$ sits inside $E(T^*X)$ as the zero section. Thus, unless $\mathcal F$
is a smooth foliation, $\mathrm{ch}(\mathcal F)$ is always singular. It follows promptly from the above local expression of
$\mathrm{ch}(\mathcal F)$ that its singular points away from the zero section and over a fiber $\pi^{-1}(p)$ are the $1$-forms
at $T^*_p X$ which annihilates the image of $D \xi (p)$. Thus, if the singular scheme of $\mathcal F$ is reduced
and of dimension zero then $\mathrm{ch}(\mathcal F)$ is smooth away from the zero section.

\subsection{Prolongation} Recall that $T^* X$ is endowed with a canonical symplectic structure which, in the above local
coordinates, is induced by the $2$-form $$\Omega = \sum dx_i \wedge dy_i.$$ If $F$ is a  holomorphic
function on (an open subset of) $T^*X$ then the hamiltonian of $F$ is by definition the vector field $\xi_F$ determined by
the formula
\[
   d F(\cdot)  = \Omega(\xi_F, \cdot) \, .
\]
Notice that the vector field $\xi_F$ is tangent to the hypersurface determined by $F$ since  $\xi_F( F) = 0$.
  Leibniz's rule implies that  $\xi_{uF} = u\xi_F + F \xi_u$.Consequently the restriction
of the direction field determined by $\xi_F$ to $\{ F=0\}$ is the same as the one of determined by $\xi_{uF}$ for an arbitrary unit $u$.
Therefore, the symplectic structure determines a one-dimensional foliation on  any reduced and irreducible hypersurface $H \subset T^*X$:
one has just to  factor
out eventual divisorial components of the singular set of  ${\xi_F}_{|H}$
to end up with a foliation on $H$, usually called in the literature the characteristic foliation of $H$. When $H = \mathrm{ch}(\mathcal F) \subset T^*X$,
 is the characteristic variety of a foliation $\mathcal F$ on $X$  we will denote  its characteristic foliation by $\mathcal F^{(1)}$ and
call it  the first prolongation  of $\mathcal F$.

If $U \subset X$ is an open set with coordinates as in \S\ref{SS:CV} and  $\xi = \sum a_i \partial_{x_i}$ is a  vector field
inducing $\mathcal F$ on $U$ then the vector field
\begin{equation}\label{E:explicito}
\hat{\xi} = \sum_{i=1}^n a_i \partial_{x_i} -  \sum_{i,j=1}^n \left(\partial_{x_j} a_i \right) y_i \partial_{y_j}
\end{equation}
is the hamiltonian vector field of $\sum a_i y_i$, and hence defines the prolongation of $\mathcal F_{|U}$.

\section{Warm-up: Proof of Theorem \ref{T:1} in  dimension three}

In this section we present a proof of Theorem \ref{T:1} in dimension three. We believe this will make the general case easier to understand.

\subsection{Making sense of the $\mathcal F^{(1)}$-invariance} We start by clarifying the meaning of $\mathcal F^{(1)}$-invariance.
The first result is well-known and holds in arbitrary dimension.

\begin{lemma}
If $Y \subset \mathrm{ch}(\mathcal F)$ is  $\mathcal F^{(1)}$-invariant then $\pi(Y)$ is $\mathcal F$-invariant.
\end{lemma}
\begin{proof}
If $p$ is a smooth point of $\mathrm{ch}(\mathcal F)$ then equation (\ref{E:explicito}) makes
clear that $\pi$ sends  $T_p\mathcal F^{(1)}$ into $T_\pi(p) \mathcal F$, and that the restriction of
$\mathcal F^{(1)}$ to the zero section is nothing more than $\mathcal F$. Together  these two facts promptly imply
the lemma.
\end{proof}

Our next result holds only in dimension three, and it is the lack of a direct analogue in higher dimensions
which will make the proof in the general case more involved.

\begin{prop}\label{P:web}
Suppose $n=3$ and let  $Y \subsetneq \mathrm{ch}(\mathcal F)$ be a homogenous, and irreducible subvariety with dominant projection to $X$.
If $Y$ is $\mathcal F^{(1)}$-invariant  then $\mathcal F$ is tangent to a codimension one web $\mathcal W_{Y}$ on $X$.
\end{prop}
\begin{proof}
Since we are in dimension three, over the smooth locus of $\mathcal F$,  $\mathrm{ch}(\mathcal F)$ is a rank two vector subbundle of
$\Omega^1_X$. A subvariety $Y$ as in the statement, determines  $k$ distinct lines on
$N^* \mathcal F_x$ for  generic points $x \in X$. Therefore $Y$ can be seen as the graph of  a rational section $\varpi$ of $Sym^k \Omega^1_X$. Moreover, the foliation $\mathcal F$ is tangent to the multi-distribution determined
by $\varpi$. Notice that so far, we have not used  the $\mathcal F^{(1)}$-invariance of $Y$, we just explored the fact that
$Y$ is contained in $\mathrm{ch}(\mathcal F)$.

It remains to prove the integrability of the multi-distribution determined by $\varpi$. For that sake we can place ourselves at a neighborhood of a point $x \in X$
where $\varpi$ is holomorphic and equal to the product of $k$ pairwise distinct $1$-forms, say $\omega_1, \ldots, \omega_k$,
and $\mathcal F$ is smooth. Choose a local coordinate system $(x_1, \ldots, x_n)$ where $\mathcal F$ is induced
by the vector field $\xi = \partial_{x_1}$. Hence $\mathcal F^{ (1)} $ is still induced by $\partial_{x_1}$ now
seen as a vector field on the total space of $N^* \mathcal F$.

If $\omega$ is any of the $1$-forms $\{ \omega_i \}_{i \in \underline k}$
then  $\omega = a dx_2 + b dx_3$ for
suitable holomorphic functions $a, b$. Notice that $\omega$ is integrable if and only if the quotient $a/b$ does not depend on
$x_1$. Finally, the $\mathcal F^{(1)}$-invariance  of $Y$ ensures that $a/b$ is constant along the orbits of $\hat{\xi}$ and
thus $\omega$ is integrable and so is the multi-distribution induced by $\varpi$.
\end{proof}

\subsection{Invariant subvarieties from singular points}

\begin{prop}\label{P:aa}
Let $\mathcal F$ be a foliation on $X$ a smooth projective variety of dimension three.
Suppose $\mathcal F$ is tangent to a codimension one web $\mathcal W$. If $p \in \mathrm{sing}(\mathcal F)$ is an
isolated singularity then there exists an   irreducible
$\mathcal F$-invariant subvariety $Y \subsetneq X$ of positive dimension containing $p$.
\end{prop}
\begin{proof}
Suppose $\mathcal W$ is a $k$-web with $k\ge 1$.
If $k \ge 2$, let $\Delta(\mathcal W) \subset X$ be the discriminant of the web $\mathcal W$. By definition, $\Delta(\mathcal W)$ is the
set  where $\mathcal W$ is not the product of $k$ pairwise transverse foliations. The
proof of Proposition \ref{P:web} tell us that on a neighborhood of a  smooth point of $\mathcal F$, the web
$\mathcal W$ is induced by a $k$-symmetric $1$-form $\varphi = \sum a_{ij} dx_2^i dx_3^j$ . Thus $\Delta(\mathcal W)$
is defined as the hypersurface cut out by the discriminant of $\varpi$, seen as a the binary form on
the variables $dx_2, dx_3$. Notice that $\Delta(\mathcal W)$ is a $\mathcal F$-invariant hypersurface.

If $p$ belongs to $\Delta(\mathcal W)$ we are done. Otherwise $\mathcal W$, at a neighborhood of $p$, can
be written as the superposition of $k$  foliations, that is $\mathcal W= \mathcal G_1 \boxtimes \cdots
\boxtimes \mathcal G_k$.   So consider one foliation $\mathcal G$ of codimension one
in a neighborhood of $p$ and suppose that $\mathcal F$ is tangent to it.

Let $\xi$ be  holomorphic vector field inducing $\mathcal F$  and $\omega$ be a holomorphic $1$-form inducing
$\mathcal G$ both defined on a neighborhood of $p$ and  without divisorial components in their zero sets.
Since $\mathcal F$ has an isolated singularity at $p$ so does $\xi$.
Consequently, $\omega(\xi)=0$ implies that
$\omega$ is also singular at $p$. At this point we can use an argument laid down by Cerveau in \cite[page 46]{cerveau} that
we now recall. As $\xi$ has isolated singularities we can apply
De Rham-Saito Lemma to ensure the existence of another vector field $\zeta$
such that $\omega = i_{\xi} i_{\zeta} dx\wedge dy \wedge dz$. Therefore the zero set of $\omega$ is formed
by the minors of a $3 \times 2$ matrix and must be of codimension at least two.
But if $\mathcal G$ is one of the foliations $\mathcal G_i$ then $\mathrm{sing}(\mathcal G)$ is algebraic, and
is the sought $\mathcal F$-invariant variety.
\end{proof}

\subsection{Conclusion of the proof} To conclude the proof of Theorem \ref{T:1} in dimension three we will make use of the
following generalization of Jouanolou's Theorem proved in \cite{CP}, see also \cite[Theorem 2]{Lu} for the very same statement on projective spaces.

\begin{thm}\label{T:CP}
Let $ X$ be a smooth projective variety,  $\mathcal L$ be an ample line bundle over it, and $k \gg 0$ be a sufficiently large
integer. If $ \mathcal F \in \mathbb P H^0(X, TX \otimes \mathcal L^{\otimes k})$ is a very generic foliation then, besides $X$ itself,
the only  subvarieties left invariant by $\mathcal F$ are its singular points.
\end{thm}

Let $\mathcal F \in \mathbb P H^0(X, TX \otimes \mathcal L^{\otimes k})$  be a very generic foliation without
invariant subvarieties. As its singular set has cardinality given by the top Chern class of $TX \otimes \mathcal L^{\otimes k}$, and
this number is positive for $k \gg 0$, the singular set of $\mathcal F$ is non-empty. Moreover we can assume the existence of
an isolated singularity  $p \in \mathrm{sing}(\mathcal F)$, see for instance \cite[Proposition 2.4]{CP}.

If the characteristic variety of $\mathcal F$  is not quasi-minimal then Proposition \ref{P:web} implies
that $\mathcal F$ is tangent to a codimension one web $\mathcal W$. Proposition \ref{P:aa} in its turn
implies that $\mathcal F$ has a invariant subvariety through $p$. This contradicts Theorem \ref{T:CP} and
concludes the proof of Theorem \ref{T:1} in dimension three. \qed

\subsection{Obstructions to generalize} To  generalize the argument above to deal with the general case  one has to circumvent the following obstructions:
\begin{enumerate}
\item Proposition \ref{P:web} does not generalize because irreducible components of $ch(\mathcal F)$ which are homogenous and dominate the base $X$ are no longer graphs of multi-distributions as happens in the three dimensional case; and
\item Proposition \ref{P:aa} does not generalize since (multi)-distributions with infinitesimal automorphisms are not necessarily integrable.
\end{enumerate}

To accomplish that  we will take advantage of the   structure  of generic
foliation singularities combined with the following reinterpretation of Theorem \ref{T:CP}.

\begin{thm}\label{T:CPbis}
Let $ X$ be a smooth projective variety,  $\mathcal L$ an ample line bundle over it, and $k \gg 0$ a sufficiently large
integer. If $ \mathcal F \in \mathbb P H^0(X, TX \otimes \mathcal L^{\otimes k})$ is a very generic foliation then
every leaf of $\mathcal F$ is Zariski dense.
\end{thm}

\section{Prolongation versus holonomy}

In this section  $\mathcal F$ will be a {\bf smooth}  foliation of dimension one   on
a complex manifold $X$.

\subsection{Holonomy}
To each leaf $L$ of $\mathcal F$, once a point $p \in L$ and a germ $(\Sigma,p)$ of smooth hypersurface transverse to
$\mathcal F$  are fixed, one can associate a (anti)-representation
\[
hol(L) : \pi_1( L, p) \longrightarrow \mathrm{Diff}(\Sigma,p)  \, ,
\]
as follows. Given a  closed path $\gamma$ contained in $L$ and centered at $p$ one defines a germ diffeomorphism
$h_{\gamma} \in   \mathrm{Diff}(\Sigma,p)$ such that $h_{\gamma}(x)$ is  the end point of a lift of $\gamma$ to the leaf of $\mathcal F$ through
$x$.  The result does not
depend on the choices involved in the process and is completely determined by the class of $\gamma$ in $\pi_1(L,p)$.  Thus one set
$hol(L)(\gamma) = h_{\gamma}$. It is an anti-representation since
$h_{\gamma_1 \cdot \gamma_2} = h_{\gamma_2} \circ h_{\gamma_1}$.

Of course, one can also consider the linear holonomy of $L$ which is just the anti-representation
\begin{align*}
Dhol(L) : \pi_1( L, p) &\longrightarrow GL(T_p \Sigma)   \, \\
[\gamma] &\longmapsto Dh_{\gamma}(p) \, .
\end{align*}

Being a anti-representation of $\pi_1(L)$ onto a general linear group it is natural to
wonder if there is a natural connection on a natural vector bundle over $L$ inducing. It is indeed
the case, and even better, there is a partial connection along the tangent bundle $T \mathcal F$ of $\mathcal F$
on the normal bundle $N\mathcal F$ which has monodromy along the leaves of $\mathcal F$ equivalent to the linear
holonomy.

\subsection{Bott's partial connection}
Let $\rho: TX \to N \mathcal F$ be the natural projection. Of course $ker \rho = T \mathcal F$. Bott's partial connection
is defined as follows
\begin{align*}
\nabla : T \mathcal F &\longrightarrow Hom(N\mathcal F, N \mathcal F) \simeq N^* \mathcal F \otimes N \mathcal F \\
\xi &\longmapsto \{ \vartheta \mapsto \rho ( [ \hat{\vartheta}, \xi ] ) \} \, ,
\end{align*}
where $\hat{\vartheta}$ stands for an arbitrary lift of $\vartheta$ to $TX$. The involutiveness  of $T\mathcal F$
implies that $\rho ( [ \hat{\vartheta}, \xi ] )$ does not depend on the choice of the lift, and ensures that $\nabla$
is well defined.

Let us now proceed to write explicitly the restriction of $\nabla$ to a leaf $L$ of
$\mathcal F$. We will work in local coordinates $(x_1,x_2,\ldots, x_{n})$ and will assume
that $L = \{ x_2= \ldots = x_{n} = 0 \}$. Since $L$ is invariant by $\mathcal F$, we can
write a vector field $\xi$  generating $T\mathcal F$ in the following form
\[
\xi = a(x) \partial_{x_1} + \sum_{i=2}^{n} \sum_{j=2}^{n} a_{ij}(x) x_i \partial_{x_j} \, .
\]

Notice that vector fields  $\partial_{x_2}, \ldots, \partial_{x_n}$ can be interpreted as a basis
of $N\mathcal F$. Thus
\[
 \nabla (\xi) ( \partial_{x_i} )  = \rho \left( \partial_{x_i}a(x) \partial_{x_1} + \sum_{i=2}^{n} \sum_{j=2}^{n}  (\partial_{x_i} a_{ij}(x) ) x_i \partial_{x_j}  +
 + \sum_{j=2}^{n}  a_{ij}(x)  \partial_{x_j} \right) \, .
\]
Hence, the induced connection $\nabla_{|L} : TL \to N^* L \otimes N L$ is
\begin{align}
\nabla_{|L}(\xi) & = \sum_{i=1}^{n-1} \sum_{j=2}^{n}   a_{ij}(x_1,0) dx_i \otimes \partial_{x_j} \nonumber  \\
& = (dx_2, \ldots, dx_n) \cdot A(x_1,0) \cdot (\partial_{x_2}, \ldots, \partial_{x_n} )^T \, . \label{E:eqA}
\end{align}

\subsection{Comparison with the prolongation}In order to compare with Bott's connection, let us now write down
the restriction to $\pi^{-1}(L)$ of the  lift of $\xi$ to $E(N^* \mathcal F)$.  We will use the same system of coordinates used in Section \ref{SS:CV}, where
$y_i = \partial_{x_i}$.  Since in these coordinates $\pi^{-1}(L)   = \{ y_1 = x_2= \ldots = x_n= 0 \}$, we can write
\begin{equation*}
\hat{\xi}_{|\pi^{-1}(L)} = a(x_1,0) \partial_{x_1} - \sum_{i,j=2}^n \left(a_{ji}(x_1,0) \right) y_i \partial_{y_j}
\end{equation*}
which in matrix form is
\[
\hat{\xi}_{|\pi^{-1}(L)} = a(x_1,0) \partial_{x_1} - (y_2, \ldots, y_n) \cdot  A^T(x_1,0) \cdot (\partial_{y_2}, \ldots, \partial{y_n})^T
\]
with $A(x_1,0)$ being the same matrix as in (\ref{E:eqA}).
It is then clear, that in these coordinates, the leaves of $\mathcal F^{(1)}$ restricted to $\pi^{-1}(L)$ are flat sections
of the connection on $N^* L$ having connection matrix  $-A^T$, where $A$ is the connection matrix on $\nabla_{|L}$.  We have thus
prove the following

\begin{prop}
The leaves of $\mathcal F^{(1)}$ are flat sections of the partial connection dual to Bott's partial connection.
\end{prop}

\section{From invariant subvarieties to multi-distributions}

\subsection{Non-resonant singularities}

Let $\mathcal F$ be a germ of one-dimensional  foliation on $(\mathbb C^n,0)$. Suppose that
it has an isolated singularity at the origin. Suppose  also that the linear part $D\xi(0)$ of a vector field
$\xi$ inducing $\mathcal F$ is invertible and its eigenvalues $\lambda_1, \ldots, \lambda_n$
generate a $\mathbb Z$-module  of rank $n$. We will say that a singularity of this form is a {\bf non-resonant singularity}.

\begin{lemma}
There exists $n$ germs of $\mathcal F$-invariant smooth curves  $\gamma_i : (\mathbb C,0) \to (\mathbb C^n,0)$ with
tangents at zero determined by the eigenvectors of $D\xi(0)$.
\end{lemma}
\begin{proof} The Hadamard–-Perron theorem for holomorphic flows \cite[Chapter 2, Section 7]{YY} ensures the existence
of a pair of  invariant  manifolds intersecting transversely at the origin and  such that the restriction of the vector field to each of them has a non-resonant singularity in the Poincar\'{e} domain (Section 5 loc. cit.). Poincar\'{e} normalization Theorem (loc. cit.)  implies that the corresponding restrictions
of $\xi$ are analytically linearizable. Since separatrices of the restrictions of $\xi$ are also separatrices of $\xi$, the lemma follows.
\end{proof}

The linear holonomy along a positive
oriented path around the origin contained in $\gamma_i(\mathbb C,0)$
is induced by a linearizable matrix $A_i \in GL(n-1,\mathbb C)$ with eigenvalues $\{\exp (2 \pi i  \lambda_j /\lambda_i ) \}_{j\neq i}$.
Moreover, the $\mathbb Z$-independence of the eigenvalues implies that the Zariski closure  of the subgroup of $GL(\mathbb C^{n-1})$ generated
by $A_i$ is a maximal torus $\simeq (\mathbb C^* )^{n-1}$.

\subsection{Singularities and the holonomy of separatrices} Together with Proposition \ref{P:vero}, the proposition below will replace Proposition \ref{P:web}
in the proof of the general case of Theorem \ref{T:1}. It guarantees that
invariant subvarieties of $\mathcal F^{(1)}$ correspond to multi-distributions tangent to $\mathcal F$ as soon as $\mathcal F$
has non-resonant singularities.

\begin{prop}\label{P:multd}
Let $\mathcal F$ be a foliation on  a smooth projective variety $X$  and
let  $Y \subsetneq ch(\mathcal F)$ be an irreducible subvariety with dominant projection to $X$ distinct from the zero section.
Suppose $\mathcal F$ has non-resonant  singularity $p$   and that at least one of its  separatrices  is Zariski dense.
If $Y$ is $\mathcal F^{(1)}$ invariant then the fiber of $Y$ over a generic point of $X$
is a finite union of linear spaces of the same dimension. Consequently $\mathcal F$ is tangent to  a multi-distribution of codimension $q= \dim Y -\dim X \le \dim X -2$.
\end{prop}
\begin{proof}
First consider a point $p_0 \in X$ in the Zariski dense separatrix through $p$, and let $L$ be the leaf of $\mathcal F$ through it.
The fiber $V$ of $E(N^*\mathcal F) \simeq ch(\mathcal F) \to X$ over
$p$ is a vector space of dimension $n-1$. The intersection of  $Y$ with $V$ is a subvariety of $V$ invariant by the subgroup  $G\subset GL(V)$
image of the representation
$\pi_1(L) \to GL(V)$ dual to the linear holonomy of $L$. Since $V\cap Y$ is algebraic, not only $G$ but also its Zariski closure leaves $V\cap Y$
invariant. By hypothesis, $\overline G \simeq (\mathbb C^*)^{n-1}$ is a maximal torus in $GL(V)$. Consequently $V\cap Y$ is a finite union of linear spaces
for an arbitrary $p \in L$. To be a finite union of linear subspaces is clearly a Zariski closed condition. Thus the same will hold
true for the fibers of $Y$ over points in the Zariski closure of $L$ which is, by assumption, equal to $X$.
\end{proof}

\section{From multi-distributions to invariant subvarieties}

We now proceed to establish the result which will replace Proposition \ref{P:aa}. We start with a simple lemma.

\begin{lemma}\label{L:normal}
Let $\omega \in \Omega^q = \Omega^q(\mathbb C^n) \otimes \mathbb C[[ x_1, \ldots, x_n]]$ be a
formal $q$-form. If $\omega$ is invariant by the natural $(\mathbb C^*)^n$-action on $\mathbb C^n$ then
\begin{equation}\label{E:normal}
\omega = f \cdot \left( \sum_{I\in \{1, \ldots, n\}^q} \lambda_I \frac{dx_{I}}{x_I} \right)
\end{equation}
where $f \in  \mathbb C[[ x_1, \ldots, x_n]]$,  $\lambda_I \in \mathbb C$ and
$
\frac{dx_I}{x_I} = \frac{dx_{i_1}}{x_{i_1}} \wedge \cdots \wedge \frac{dx_{i_q}}{x_{i_q}} \, .
$
\end{lemma}
\begin{proof}
Write $\omega = \sum_{i=i_0}^{\infty} \omega_i$,
where the coefficients of $\omega_i$ are polynomials of degree $i$ and $\omega_{i_0} \neq 0$.
If $\varphi_t(x) = t\cdot x$ then
\[
\frac{(\varphi_t)^* \omega}{t^{i_0 + q}} = \omega_{i_0} + \sum_{i=i_0 + 1}^{\infty} t^{ i - t_0 + q} \omega_i \, .
\]
Since for arbitrary $t$, $\varphi_t^* \omega$ must be a multiple of $\omega$ then after dividing by a suitable
formal function we can assume that $\omega$ is homogeneous.

Let $x^J dx_I$ be a monomial appearing in $\omega$. Suppose $x_1^{j_1}$ divides $x^J$ but $x_1^{j_1+1}$ does not. Consider the automorphism
$\varphi_t(x_1,x_2, \ldots, x_n) = (tx_1, x_2, \ldots, x_n)$. Then $\varphi_t^* (x^J dx_I ) = t^{j_1 + \epsilon} x^J dx_I$,
where $\epsilon=0$ if $dx_1$ does not appear in $dx_I$ and $\epsilon=1$ otherwise.  If $j_1 + \epsilon \ge 2$ then $x_1$ divides all the other
monomials appearing in $\omega$. Does after division we can assume $j_1 + \epsilon =1$  and the same will hold true for any other
monomial appearing in $\omega$. Repeating the argument for the other coordinate functions makes clear the assertion of the lemma.
\end{proof}

\begin{prop}\label{P:vero}
Let $\xi$ be a germ of holomorphic vector field on  $(\mathbb C^n,0)$ with a non-resonant singularity at the origin. Suppose
$\xi$ is an infinitesimal automorphism of  a distribution  $\mathcal D$ of codimension $q \le n-2$. Then  $\mathcal D$ is integrable
and the singular set of $\mathcal D$ has positive dimension.
\end{prop}
\begin{proof}
Let  $\omega$ be a germ of holomorphic $q$-form, $q=n-p$, defining $\mathcal D$, that is $\mathcal D = \{ v \in T (\mathbb C^n,0) \, | \, \omega(v) =0 \} $.
For further use let us recall that a $q$-form  $\omega$ defines a codimension $q$ distribution if and only if
\[
(i_v \omega ) \wedge  \omega = 0 \quad \text{ for every } \quad v \in \bigwedge^{q-1} \mathbb C^{n} \, ,
\]
and this distribution is integrable if and only if
\[
(i_v \omega ) \wedge  \omega = (i_v \omega ) \wedge  d\omega =0 \quad \text{ for every } \quad v \in \bigwedge^{q-1} \mathbb C^{n} \, ,
\]
see \cite{Medeiros}. It follows that  integrability is a formal condition, and as such can be verified in an arbitrary formal coordinate system.

Since the origin is a non-resonant singularity for $\xi$, we can choose formal coordinates such that
\[
\xi = \sum_{i=1}^n \lambda_i x_i \partial_{x_i} \,
\]
where $\lambda_i \in \mathbb C$ are complex numbers.
In exchange we can no longer  assume that
$\omega$ is a holomorphic $q$-form, but it is certainly a formal $q$-form.

Since $\xi$ is an infinitesimal automorphism of $\mathcal D$, its flow  $\varphi_t : (\mathbb C^n,0) \to (\mathbb C^n,0)$
preserves $\omega$. More precisely,
\[
\varphi_t^* \omega = f(t,x) \omega
\]
for a suitable formal function $f \in \mathbb C[[ t, x_1, \ldots, x_n]]$.

Consider now the subgroup $G \subset (\mathbb C^*)^n \subset GL(\mathbb C^n)$ defined as
\[
G = \{ A \in (\mathbb C^*)^n  \, | \, A^* \omega \wedge \omega = 0 \text{ in } \bigwedge^2 \Omega^q \otimes \mathbb C [[ x_1, \ldots, x_n ]] \} \, ,
\]
where $(\mathbb C^* )^n$ acts on $(\mathbb C^n,0)$ through a diagonal linear map. The flow of $\xi$ determines a non-closed one parameter
subgroup of $H \subset G$. Since  $G$ is clearly an algebraic subgroup, it follows that the Zariski closure of $H$
 is also contained in $G$. But the dimension of the Zariski closure of $H$ is nothing more than the rank of the $\mathbb Z$-module
 generate by $\lambda_1, \ldots, \lambda_n$. It follows that $\overline H = G = (\mathbb C^* )^n$.

On the one hand, since $\omega$ induces a distribution $\i_v \omega \wedge \omega =0$. On the other hand, Lemma \ref{L:normal} implies
that $\omega$ is a multiple of a closed $q$-form, and consequently $\i_v \omega \wedge d\omega =0$. This shows that $\mathcal D$ is integrable.

\smallskip

It remains to prove that the singular set of $\mathcal D$ has positive dimension. Looking at the expression (\ref{E:normal}) we
realize that it must have at least two non-trivial summands. Indeed, if not, $\mathcal D$ would be a smooth foliation tangent
to $\xi$, what is clearly impossible. Therefore,  if  $k$ is the cardinality of the set
$\overline I = \cup_{\lambda_I \neq 0} I $, where the complex numbers  $\lambda_I$ are defined by (\ref{E:normal}),  then $k>q$. Clearly,
the coordinate hyperplanes with index in $\overline I$ is invariant by $\mathcal D$. Consequently the intersection of any $q+1$ of these coordinates
hyperplanes is also invariant
by $\mathcal D$. Since $\mathcal D$ has codimension $q$, this intersection must be contained in the singular locus of $\mathcal D$.
\end{proof}

\begin{remark}
Proposition \ref{P:vero} will be in the proof of the general case of Theorem \ref{T:1} what Proposition \ref{P:web} is in the proof of
the three-dimensional case. The analogy is not perfect as we do not proved here the integrability of multi-distributions as we did there.
Anyway, with some extra effort one can also prove the integrability of the multi-distribution. We will not pursue this here as the result above
is sufficient for our purposes.
\end{remark}

\section{Proof of Theorem \ref{T:1}}

Let $\mathcal F \in \mathbb P H^0(X,TX \otimes \mathcal L^{\otimes k})$ be a very generic foliation. We can assume, thanks
to Theorem \ref{T:CPbis},   that $\mathcal F$
has isolated singularities, at least one non-resonant singularity, and every leaf of $\mathcal F$ is Zariski dense.

Proposition \ref{P:multd} implies that $\mathcal F$ is tangent to a multi-distribution $\mathcal D$.
We can assume $\mathcal D$ is irreducible without loss of generality.

If $\mathcal D$ is locally decomposable  around $p$ then  Proposition \ref{P:vero} implies
the existence of a positive dimensional irreducible component $Z$ of the singular set of $\mathcal D$ through $p$. This set
is clearly algebraic and  invariant by $\mathcal F$ since sections of $T\mathcal F$ are infinitesimal automorphisms of $\mathcal D$.
If $\mathcal D$ is not locally decomposable at $p$  then there exists
a subvariety $Z \subsetneq X$ where $\mathcal D$ is not locally decomposable. As above, we conclude that  $Z$ is  invariant by $\mathcal F$.

In both cases, we arrive at a contradiction with Theorem \ref{T:CPbis}.  \qed

\section{Bernstein-Lunts Conjecture}\label{S:affine}

Theorem \ref{T:1} implies  the existence of foliations, on arbitrary projective varieties, with quasi-minimal characteristic
variety. Moreover, as the conclusion of Theorem \ref{T:CPbis} holds true for any foliation with ample cotangent bundle on $\mathbb P^n$,
the existential part of Bernstein-Lunts Conjecture is settled. Nevertheless, there is still a detail to be dealt with in order to prove that a
\emph{very generic} polynomial vector field  of degree $d\ge 2$ has quasi-minimal characteristic variety.

\subsection{Projective versus affine degree}

The (projective) degree of a holomorphic foliation $\mathcal F$ on $\mathbb P^n$  is defined as the
degree of the tangency divisor of $\mathcal F$ with a generic hyperplane $H$. If $\mathcal F \in \mathbb P H^0(\mathbb P^n, \mathcal O_{\mathbb P^n}(k))$
then the degree of $\mathcal F$ is equal to $k+1$.

If one starts with a polynomial vector field $\xi$  of degree $d$ on $\mathbb C^n$ then it is natural to extend it to
holomorphic foliation $\mathcal F_{\xi}$  on $\mathbb P^n$ such that $H$ is not contained in the singular set of $\mathcal F_{\xi}$.
We set the degree of $\xi = \sum a_i \partial_i$ as the maximal degree of its
coefficients $a_i$. In general the (projective) degree of $\mathcal F_{\xi}$ is at most the (affine) degree of $\xi$. Moreover precisely,
\[
\deg (\mathcal F_{\xi} ) = \left\{ \begin{array}{ll}
                                     \deg(\xi) & \text{ if } H \text{ is invariant by }  \mathcal F_{\xi}, \\
                                     \deg(\xi)-1  &\text{ if } H \text{ is not invariant by  }  \mathcal F_{\xi}.
                                   \end{array} \right.
\]
If $\mathcal D(n,d)$ is the set of polynomial vector fields of degree at most $d$ then the generic element in it
extends to a foliation of $\mathbb P^n$ with singularities of codimension at least two which leaves the hyperplane
at infinity invariant, see \cite{Z} for a through discussion. In more intrinsic terms, if $T \mathbb P^n ( - \log H)$ denotes the subsheaf of $T\mathbb P^n$ constituted by
germs of vector fields tangent to $H$ then $\mathcal D(n,d)$ can be identified with $H^0(\mathbb P^n, T \mathbb P^n (- \log H) \otimes \mathcal O_{\mathbb P^n}(d-1))$.
Under this identification the extension which do not leave the hyperplane at infinity invariant will appear with a divisorial component in their singular set supported
there.

\subsection{Relative version of Theorem \ref{T:CPbis}}
The proof of Theorem \ref{T:CPbis}  can be adapted to prove the following

\begin{thm}\label{T:CPbisrel}
Let $ X$ be a smooth projective variety and  $H \subset X$ a smooth hypersurface. Let also  $\mathcal L$ be an ample line bundle over $X$, and $k \gg 0$ a sufficiently large integer. If $ \mathcal F \in \mathbb P H^0(X, TX (- \log H) \otimes \mathcal L^{\otimes k})$ is a very generic foliation then
every leaf of $\mathcal F$ not contained in $H$ is Zariski dense.
\end{thm}

We will not detail its proof as the case of  projective spaces (the one used in the proof of Theorem \ref{T:2} below)  is
Theorem 4.2 of \cite{Coutinho}. Moreover, there it is proved that it suffices to take $k\ge 1$ when $X = \mathbb P^n$ and $\mathcal L= \mathcal O_{\mathbb P^n}(1)$.

\subsection{Proof of Theorem \ref{T:2}}

According to Theorem \ref{T:CPbisrel} the leaves of a very generic vector field of degree $d \ge 2$ are Zariski dense. Also a
very generic vector field  has  at least one non-resonant singularity. Thus we can apply Propositions \ref{P:multd} and \ref{P:vero}
to conclude that the characteristic variety of $\mathcal F_{\xi}$ is quasi-minimal. \qed

\end{document}